\documentclass{article}
\usepackage{amsmath, amssymb, amsthm, graphicx, amscd, amsfonts, amscd}

\usepackage[pagewise, displaymath, mathlines]{lineno}

\usepackage{color}   
\usepackage{hyperref}
\hypersetup{
    colorlinks=true, 
    linktoc=all,     
    linkcolor=blue,  
}

\usepackage{tikz-cd}

\newtheorem{theorem}{Theorem}[section]

\newtheorem{proposition}[theorem]{Proposition}
\newtheorem{corollary}[theorem]{Corollary}

\def\bb #1{ {\mathbb #1} }
\def\c #1{ {\mathcal #1} }

\def\deg { \text{deg} }

\begin{document}
\title{Kloosterman sums and Hecke polynomials \\ in characteristics 2 and 3}

\author{C. Douglas Haessig\footnote{This work was partially supported by a grant from the Simons Foundation \#314961.}}

\maketitle

\abstract{In this paper we give a modular interpretation of the $k$-th symmetric power $L$-function of the Kloosterman family of exponential sums in characteristics 2 and 3, and in the case of $p=2$ and $k$ odd give the precise 2-adic Newton polygon. We also give a $p$-adic modular interpretation of Dwork's unit root $L$-function of the Kloosterman family, and give the precise 2-adic Newton polygon when $k$ is odd.

In a previous paper, we gave an estimate for the $q$-adic Newton polygon of the symmetric power $L$-function of the Kloosterman family when $p \geq 5$. We discuss how this restriction on primes was not needed, and so the results of that paper hold for all $p \geq 2$.}

\section{Introduction}

Fix $p$ a prime number and a primitive $p$-th root of unity $\zeta_p$. For each $t \in \overline{\bb F}_p^*$ and $m \in \bb Z_{\geq 1}$ we define the Kloosterman exponential sum
\[
Kl_m(t) := \sum_{x \in \bb F_{q^m}^*} \zeta_p^{Tr_{\bb F_{q^m} / \bb F_p}(x + \frac{t}{x})} \in \bb R
\]
where $q = p^{\deg(t)}$ and $\deg(t) := [\bb F_p(t) : \bb F_p]$. It is well-known that the associated $L$-function is a quadratic polynomial 
\[
L(t, T) := \exp\left( \sum_{m=1}^\infty Kl_m(t) \frac{T^m}{m} \right) = (1 - \pi_0(t) T)(1 - \pi_1(t) T)
\]
whose reciprocal roots satisfy $| \pi_0(t) | = | \pi_1(t) | = \sqrt{p^{\deg(t)}}$ and $\pi_0(t) \pi_1(t) = p^{\deg(t)}$. For $k$ a positive integer, define the $k$-th symmetric power $L$-function of the family by 
\[
L(Sym^k Kl, s) := \prod_{t \in | \bb G_m / \bb F_p |} \prod_{i = 0}^k \frac{1}{1 - \pi_0(t)^{k-i} \pi_1(t)^i s^{\deg(t)}},
\]
where the first product runs over all closed points of the algebraic torus defined over $\bb F_p$. The main result of this paper is to give a modular interpretation of this $L$-function in the cases $p = 2$ and 3. To state the result we need some notation. For a congruence subgroup $\Gamma$ of $SL_2(\bb Z)$ and character $\chi$, denote the space of cuspforms of level $\Gamma$ and weight $k$ by $S_k(\Gamma, \chi)$, and $S_k(\Gamma)$ if $\chi$ is trivial. Denote by $S_k^\text{new}(\Gamma, \chi)$ the new subspace of $S_k(\Gamma, \chi)$. Let $T_n$ be the $n$-th Hecke operator and $U_p$ Atkin's $U$-operator.

\begin{theorem}\label{T: general}
For $p = 2$ set $\Gamma := \Gamma_1(4)$, and for $p = 3$ set $\Gamma := \Gamma_1(3)$. Then for every $k \geq 1$
\[
L(Sym^k Kl, s) = (1 - s) det(1 - s U_p \mid S_{k+2}(\Gamma)).
\]
Writing $L(Sym^k Kl, s) = \sum a_m s^m$ then the $p$-adic Newton polygon satisfies $ord_p \> a_m \geq m(m-1)$. If $p = 2$ and $k$ is odd, then this is an equality, $ord_p \> a_m = m(m-1)$.
\end{theorem}

\begin{corollary}\label{refined}
Let $p = 2$ or $3$. 
\begin{enumerate}
\item Suppose $k$ is even. Over $\bb Z$, we may factor $L(Sym^k Kl, s) = (1 - s) \cdot P_k(s) \cdot M_k(s)$ where
\begin{align*}
M_k(s) &=  det(1 - s T_p + p^{k+1} s^2 \mid S_{k+2}( SL_2(\bb Z))) \\
P_k(s) &= det (1 - s U_p  \mid S_{k+2}^\text{new}(\Gamma_0(p))) = (1 + p^{k/2} s)^{a_k} \cdot (1 - p^{k/2} s)^{b_k}.
\end{align*}
Further, $M_k$ is pure of weight $k+1$ and satisfies the functional equation $M_k(s) = c s^d M_k(1/p^{k+1} s)$ where $d := \text{deg } M_k$ and $c$ is a constant. The constants $a_k$ and $b_k$ are given in \cite{MR3305309} when $p = 2$, and in \cite{FuWan-$L$-functionssymmetricproducts-2005} when $p = 3$.
\item Suppose $k$ is odd. Then all reciprocal roots of $L(Sym^k Kl, s) = (1 - s) det(1 - s U_p  \mid S_{k+2}(\Gamma))$ except for $s = 1$ are pure of weight $k+1$.
\end{enumerate}
\end{corollary}

For $p \geq 5$, we do not know of such a modular interpretation, however, the factorization into $P_k$ and $M_k$ is known. That is, from the work of \cite{MR2425148, Robba-SymmetricPowersof-1986, MR3305309} it is known that for all $p$ and $k$ we may factorize $L(Sym^k Kl, s) = (1 - s) \cdot P_k(s) \cdot M_k(s)$ with $P_k(s) = \prod (1 \pm p^{k/2} s)$ and $M_k$ pure of weight $k+1$ satisfying the same functional equation as above. A summary of this may be found in \cite[Section 5]{Fresan2018HodgeTO}. 

When $p \geq 3$ the same lower bound $ord_p \> a_m \geq m(m-1)$ for the Newton polygon was proven in \cite{Fresan2018HodgeTO}, which was an improvement to \cite{MR3694643}. However, in Section \ref{S: padicSym} we discuss how the theory in  \cite{MR3694643} holds for all $p \geq 2$ without change, including the improved lower bound. We expect the precise Newton polygon to be difficult to describe. An easier task would be better lower bounds for the Newton polygon. For example, for $p = 3$ my undergraduate student Guowei Shan conjectures that the 3-adic Newton polygon of $det(1 - s U_3 \mid S_{k+2}(\Gamma_1(3))$ of $L(Sym^k Kl, s)$ is bounded below by $y = (1/2)(3x^2 +x)$.

Symmetric powers played a pivotal role in Wan's proof of Dwork's meromorphy conjecture for unit root $L$-functions \cite{Wan-DworkConjectureunit-1999, Wan-Higherrankcase-2000, Wan-Rankonecase-2000}. The \emph{Kloosterman unit root $L$-function} is defined as follows. Let $\bb Z_p$ denote the $p$-adic integers, and let $\kappa \in \bb Z_p$. For each $t \in \overline{\bb F}_p^*$, the root $\pi_0(t)$ is a $p$-adic 1-unit, and so we may define
\[
L_\text{unit}(\kappa, T) := \prod_{t \in | \bb G_m / \bb F_p|} \frac{1}{1 - \pi_0(t)^\kappa T^{deg(t)}}.
\]
To prove meromorphy, Wan introduced the $p$-adic $\kappa$-symmetric power $L$-function $L(Sym^{\infty, \kappa} Kl, s)$, defined as follows. For a sequence $( k_i )$ of positive integers tending to infinity, and such that $k_i \rightarrow \kappa$ $p$-adically, the limit $\lim_{i \rightarrow \infty} L(Sym^{k_i} Kl, s) = L(Sym^{\infty, \kappa} Kl, s)$ exists and is a formal power series with $p$-adic integral coefficients. In \cite{MR3249829}, it was shown to have the Euler product
\[
L(Sym^{\infty, \kappa} Kl, s) = \prod_{t \in | \bb G_m / \bb F_p |} \prod_{i = 0}^\infty \frac{1}{1 - \pi_0(t)^{k-i} \pi_1(t)^i s^{\deg(t)}}.
\]
In \cite{MR3694643} a $p$-adic cohomology theory was given for $L(Sym^{\infty, \kappa} Kl, s)$ when $p \geq 5$, however, in Section \ref{S: padicSym} we show how it holds for all $p \geq 2$:
\[
L(Sym^{\infty, \kappa} Kl, s) = \frac{det(1 - s \beta_\kappa \mid H_\kappa^1)}{det(1 - p s \beta_\kappa  \mid H^0_\kappa)},
\]
where $\beta_\kappa$ is a completely continuous operator defined on $p$-adic cohomology spaces $H^1_\kappa$ and $H^0_\kappa$ which satisfy: if $\kappa \not= 0$ then $H^0_\kappa = 0$, else if $\kappa = 0$ then $H^0_\kappa$ is 1-dimensional. As a consequence we get the following $p$-adic modular interpretation of $L(Sym^{\infty, \kappa} Kl, s)$. Denote by $S^\dag_\kappa(\Gamma)$ the space of overconvergent $p$-adic cusp forms of level $\Gamma$ and weight $\kappa$. 

\begin{theorem}\label{T: overconv interp}
For $p = 2$ set $\Gamma := \Gamma_1(4)$, and for $p = 3$ set $\Gamma := \Gamma_1(3)$. For $\kappa \in \bb Z_p \setminus \{0\}$, 
\[
L(Sym^{\infty, \kappa} Kl, s) = (1 - s) det(1 - s U_p \mid S^\dag_{\kappa+2}(\Gamma)) = det(1 - s \beta_\kappa \mid H^1_\kappa).
\]
For $p = 2$ and $\kappa \in 1 + 2 \bb Z_2$, the 2-adic Newton polygon of $L(Sym^{\infty, \kappa} Kl, s)$ has 2-adic slopes $2n$ for $n \geq 0$. For $p = 3$, or $p = 2$ and $\kappa \in 2\bb Z_2$, then the slopes are $\geq 2n$ for every $n$.
\end{theorem}

The Kloosterman unit root $L$-function is related to the $p$-adic symmetric power $L$-function through the identity
\[
L_\text{unit}(\kappa, s) = \frac{ L(Sym^{\infty, \kappa} Kl, s) }{ L(Sym^{\infty, \kappa-2} Kl, ps) },
\]
and thus we get the $p$-adic modular description of the unit root $L$-function:

\begin{corollary}
Let $\Gamma$ be as in Theorem \ref{T: overconv interp}. For $\kappa \in \bb Z_p$ and $\kappa \not= 0, 2$, then
\[
L_\text{unit}(\kappa, s) = \frac{(1 - s) det(1 - s U_p \mid S^\dag_{\kappa+2}(\Gamma))}{(1 - ps) det(1 - ps U_p \mid S^\dag_{\kappa}(\Gamma))} = \frac{det(1 - s \beta_\kappa \mid H^1_\kappa)}{det(1 - ps \beta_\kappa \mid H^1_{\kappa-2})}.
\]
Further, when $p = 2$ and $\kappa \in 1 + 2\bb Z_2$, the numerator of $L_\text{unit}(\kappa, s)$ has 2-adic slopes $2n$ for $n\geq 0$ whereas the denominator has slopes $2n+1$ for $n \geq 0$; in particular, there are no cancellations among the zeros and poles.
\end{corollary}

Lastly, while we have focused on symmetric powers, other linear-algebraic operations, such as tensor or exterior products, have been studied for other families. For general multi-parameter families see \cite{MR3239170, MR3572279}; for moment $L$-functions of Calabi-Yau hypersurfaces see \cite{MR2385246}; for the hyper-Kloosterman family see \cite{FuWan-$L$-functionssymmetricproducts-2005,MR2425148,FuWan-TrivialfactorsL-functions-, MR3572279}; for the Airy family \cite{MR2542216, MR2873140}; and for the Legendre family of elliptic curves see \cite{Adolphson-$p$-adictheoryof-1976, Haessig-EllCurve}. 

{\it Acknowledgements.} I want like to thank John Bergdall, Lloyd Kilford, Robert Pollack, and Daqing Wan for helpful discussions. 
\section{Frequency of Kloosterman sums}

In this section, we look at the frequency of Kloosterman values, that is, given a fixed $s$, how many $t \in \bb F_q^*$ produce a Kloosterman sum $-Kl_q(t)$ equal to $s$. We begin by fixing some notation. Let $p$ be a prime, and set $q = p^m$. Let $\bb Q_p$ be the field of $p$-adic numbers with $\bb Z_p$ its ring of integers. Fix $\pi \in \overline{\bb Q}_p$ satisfying $\pi^{p-1} = -p$ and note $\zeta_p := 1 - \pi$ is a primitive $p$-th root of unity. Let $\bb Z_p[\pi]$ denote the ring of integers of $\bb Q_p(\pi)$. Changing notation slightly from the introduction, for $t \in \bb F_q^*$ define the Kloosterman sum
\[
Kl_q(t) := \sum_{x \in \bb F_q^*} \zeta_p^{Tr_{\bb F_q / \bb F_p}(x + \frac{t}{x})} \in \bb Z_p[\pi] \cap \bb R.
\]

\begin{proposition}\label{P: Kl congruence}
For any prime number $p \geq 3$ and $m \geq 1$, or $p = 2$ and $m \geq 2$, the Kloosterman sums satisfy $-Kl_q(t) \equiv 1$ mod $\pi^2$ and $Kl_q(t)^2 < 4q$.
\end{proposition}

\begin{proof}
Write
\begin{align*}
Kl_q(t) &= \sum_{x \in \bb F_q^*} (1-\pi)^{Tr_{\bb F_q / \bb F_p}(x + \frac{t}{x})} \\
&= \sum_{x \in \bb F_q^*} \left( 1 - Tr_{\bb F_q / \bb F_p}(x + t/x) \pi \right) \qquad \text{mod } \pi^2 \\
&= q - 1 - \pi \sum_{x \in \bb F_q^*} Tr_{\bb F_q / \bb F_p}(x + t/x) \qquad \text{mod } \pi^2 \\
&= -1 \qquad \text{mod } \pi^2,
\end{align*}
using the fact that $\sum_{x \in \bb F_q^*} x^n = 0$ if $(q-1) \nmid n$. Thus, $-Kl_q(t) \equiv 1$ mod $\pi^2$. Next, from Weil we know $| Kl_q(t) | \leq 2 \sqrt{q}$ but this cannot be an equality else $\pi^2 \mid Kl_q(t)^2$ which we have just shown is not the case.
\end{proof}

Note that when $p = 2$ or $3$ then $Kl_q(t) \in \bb Z$, and $\pi^2 = 4, -3$, respectively. Hence, Proposition \ref{P: Kl congruence} for $p = 2, 3$ says $-Kl_q(t) \equiv 1$ mod $4$ or $3$, respectively. Now, for $f \in \bb R$ define
\[
F(q, f) := \# \{ t \in \bb F_q^* \mid -Kl_q(t) = f \}.
\]
By Proposition \ref{P: Kl congruence}, in order for $F(q, f)$ to have a chance of being non-zero, we must at least have  $f \equiv 1$ mod $\pi^2$ and $f^2 < 4q$. As the following result shows, for $p = 2$ or 3, $F(q, f)$ may be recognized as the Kronecker class number. When $p \geq 5$, we do not know an alternate description of $F(q, f)$, however, there is certainly a pattern to its distribution as $f$ varies.

\begin{theorem}\label{T: HK}
Suppose $p = 3$ and $m \geq 1$, or $p = 2$ and $m \geq 2$. Let $f \in \bb Z$ satisfy $f \equiv 1$ mod $\pi^2$ and $f^2 < 4q$. Then $F(q, f) = H(f^2 - 4q)$ where $H$ is the Kronecker class number.
\end{theorem}

\begin{proof}
This result appears in a few places, such as \cite{MR1054286, MR1111555}. The one in \cite{MR1111555} goes like this. Suppose $p = 2$. In the proof of \cite[Proposition 3.1 part (4)]{MR1111555}  $-Kl_{2^m}(t) = 2^m + 1 - \# E_t(\bb F_{2^m})$ where $E_t$ is the elliptic curve with affine equation $Y^2 + XY = X^3 + tX$. Thus, for $f \in \bb Z$, $f \equiv 1$ mod 4, and $f^2 < 4 \cdot 2^m$,
\[
F(2^m, f) = \#\{ t \in \bb F_{2^m}^* \mid \# E_t(\bb F_{2^m}) = 2^m + 1 - f \} = H(f^2 - 4 \cdot 2^m),
\]
where the last equality comes from the proof of \cite[Theorem 3.3]{MR1111555}.

The proof for $p = 3$ is essentially the same except in this case $-Kl_{3^m}(t) = 3^m + 1 - \#E_t(\bb F_{3^m})$ by \cite{MarkoMoisio2008}, where $E_t$ is the elliptic curve with affine equation $Y^2 = X^3 + X^2 - t$.
\end{proof}

\section{Modular interpretation}

Before giving the proofs of Theorem \ref{T: general} and Corollary \ref{refined}, we need a precise description of the Eichler-Selberg trace formula for the Hecke operator $T_q$ on the space of cusp forms $S_k(\Gamma_1(N))$ for $N = 3$ and $4$ for $p = 3$ and $2$, respectively.

\begin{theorem}\cite[Theorem 2.5]{MR1111555}\label{T: trace for p2}
Let $p = 2$ and $k \geq 3$, and let $q$ denote a power of 2. The trace of the Hecke operator $T_q$ on the space of cusp forms $S_k(\Gamma_1(4))$ is given by
\[
Tr(T_q) = -1 - (-1)^{qk/2} \sum_f \left( \frac{\rho^{k-1} - \bar \rho^{k-1}}{\rho - \bar \rho} \right) H(f^2 - 4q)
\]
where the sum runs over $\{ f \in \bb Z \mid f^2 < 4q \text{ and } f \equiv 1 \text{ mod }4 \}$, and $\rho, \bar \rho$ are the roots of $X^2 - fX + q$.
\end{theorem}

A similar theorem is true for the Hecke operator $T_q$ on $S_k(\Gamma_1(3))$ when $p = 3$.

\begin{theorem}\label{T: trace for p3}
Let $p = 3$ and $k \geq 1$, and let $q$ denote a power of 3. The trace of the Hecke operator $T_q$ on the space of cusp forms $S_k(\Gamma_1(3))$ is given by
\[
Tr(T_q) = -1 - \sum_f \left( \frac{\rho^{k-1} - \bar \rho^{k-1}}{\rho - \bar \rho} \right) H(f^2 - 4q)
\]
where the sum runs over $\{ f \in \bb Z \mid f^2 < 4q \text{ and } f \equiv 1 \text{ mod }3 \}$, and $\rho, \bar \rho$ are the roots of $X^2 - fX + q$.
\end{theorem}

\begin{proof}
When $k$ is even, $S_k(\Gamma_1(3)) = S_k(\Gamma_0(3), 1)$ where $1$ is the trivial character, else if $k$ is odd then $S_k(\Gamma_1(3)) = S_k(\Gamma_0(3), \chi)$ where $\chi$ is the non-trivial character of $(\bb Z/3 \bb Z)^*$. The proof of the theorem follows by Proposition 2.1 and Theorem 2.2 of \cite{MR1111555}.
\end{proof}

Way now prove the main result.

\begin{proof}[Proof of Theorem \ref{T: general}]
Write
\begin{align*}
L(Sym^k Kl, s) :&= \prod_{t \in \overline{\bb F}_p^*} \prod_{i = 0}^k (1 - \pi_0(t)^i \pi_1(t)^{k-i} s^{deg(t)})^{-1/deg(t)} \\
&= \exp\left( \sum_{r \geq 1} \sum_{t \in \overline{\bb F}_p^*} \sum_{i = 0}^k (\pi_0(t)^r)^i (\pi_1(t)^r)^{k-i} \frac{s^{r deg(t)}}{r deg(t)} \right) \\
&= \exp\left( \sum_{m \geq 1} S_m \frac{s^m}{m} \right)
\end{align*}
where
\begin{align*}
S_m :&= \sum_{r \geq 1, t \in \overline{\bb F}_p^*, r deg(t) = m} \sum_{i = 0}^k (\pi_0(t)^r)^i (\pi_1(t)^r)^{k-i} \\
&= \sum_{r \geq 1, t \in \overline{\bb F}_p^*, r deg(t) = m} \frac{( \pi_0(t)^r)^{k+1} - ( \pi_1(t)^r)^{k+1}}{\pi_0(t)^r - \pi_1(t)^r}.
\end{align*}
Note that since $\pi_0(t)$ and $\pi_1(t)$ are roots of $x^2 + Kl_{p^{deg(t)}}(t)x + p^{deg(t)}$, for $r \geq 1$ and $t \in \overline{\bb F}_p^*$ such that $r deg(t) = m$, then $\pi_0(t)^r$ and $\pi_1(t)^r$ are roots of $x^2 + Kl_{p^m}(t)x + p^m$. Grouping together the $t$ such that $-Kl_{p^m}(t) = f$, by Theorem \ref{T: HK}, if $p = 3$ and $m \geq 1$, or $p = 2$ and $m \geq 2$, then we have
\begin{equation}\label{E: Sm}
\qquad S_m = \sum_{f \in \bb Z, f^2 < 4 \cdot p^m, f \equiv 1 \text{ mod } \pi^2} \left( \frac{\rho^{k+1} - \bar \rho^{k+1}}{\rho - \bar \rho} \right) H(f^2 - 4p^m)
\end{equation}
where $\rho$ and $\bar \rho$ are roots of $x^2 - f x + p^m$. As mentioned above, $\pi^2$ = 4 or $-3$ if $p = 2$ or 3, respectively. If $p = 2$ and $m =1$, then
\[
S_1 = \sum_{t \in \bb F_2^*} \frac{\pi_0(t)^{k+1} -  \pi_1(t)^{k+1}}{\pi_0(t) - \pi_1(t)}.
\]
where $\pi_0(t)$ and $\pi_1(t)$ are roots of $x^2 + x + 2$ since $Kl_2(1) = 1$.

Since $U_{p^m} = U_p^m$, for $p = 3$ and $m \geq 1$, or $p = 2$ and $m \geq 2$, from (\ref{E: Sm}) and Theorems \ref{T: trace for p2} and \ref{T: trace for p3}, we see that the trace of $U_q$ on $S_{k+2}(\Gamma)$ satisfies $Tr(U_p^m) = -1 - S_m$. For $p = 2$ and $m = 1$, since $H(-7) = 1$, we have on $S_{k+2}(\Gamma_1(4))$
\begin{align*}
Tr(U_2) &= -1 - (-1)^{k+2} \left( \frac{\rho^{k+1} - \bar \rho^{k+1}}{\rho - \bar \rho} \right) \\
&= -1 - \left( \frac{(-\rho)^{k+1} - (-\bar \rho)^{k+1}}{(-\rho) - (-\bar \rho)} \right) \\
&= -1 - S_1.
\end{align*}
Thus, for $p = $ 2 or 3,
\begin{align*}
L(Sym^k Kl, s) &= \exp\left( \sum_{m \geq 1} S_m \frac{s^m}{m} \right) \\
&= \exp\left( \sum_{m \geq 1} (-1 - Tr(U_p^m)) \frac{s^m}{m} \right) \\
&= (1-s)det(1 - s U_p \mid S_{k+2}(\Gamma)).
\end{align*}

The result on the Newton polygon follows from Theorem \ref{T: sym inf} below and the identity:
\[
L(Sym^k Kl, s) = \frac{L(Sym^{\infty, k} Kl, s)}{L(Sym^{\infty, -(k+2)} Kl, p^{k+1} s)}.
\]
\end{proof}

\begin{proof}[Proof of Corollary \ref{refined}]
We first suppose $p = 2$ and $k$ is even. By \cite[Lemma 1]{MR369263} $U_2: S_k(\Gamma_1(4)) \rightarrow S_k(\Gamma_1(2))$, and so $det(1 - s U_2 \mid S_{k}(\Gamma_1(4))) = det(1 - s U_2 \mid S_{k}(\Gamma_1(2)))$. By Atkin-Lehner-Li theory, we may decompose $S_k(\Gamma_1(2)) = W \oplus S_k^\text{new}(\Gamma_1(2))$ where
\[
W := S_k(SL_2(\bb Z)) \oplus V_2 S_k(SL_2(\bb Z)),
\]
where $V_p$ is the map on $q$-expansions $f(q) \mapsto f(q^p)$. Let $f \in S_k(SL_2(\bb Z))$ be an eigenform of $T_2$ with $T_2(f) = \lambda_f f$. Since $T_2 = U_2 + 2^{k-1} V_2$,  the matrix of $U_2$ on the vector space generated by $\{f, V_2 f \}$ is of the form
\[
\left(
\begin{array}{cc}
\lambda_f  & 1   \\
 -2^{k-1} & 0 \\
\end{array}
\right).
\]
Thus $det(1 - s U_2 \mid W) = det(1 - T_2 s + 2^{k-1} s^2 \mid S_k(SL_2(\bb Z)))$, and so by \cite{MR340258} every reciprocal root  has complex absolute value equal to $p^{(k-1)/2}$. By \cite[Theorem 3]{MR369263}, each reciprocal root of $det(1 - U_2 s \mid S_k^\text{new}(\Gamma_1(2)))$ equals $\pm 2^{(k-2)/2}$, and the number of such roots equals dim $S_k^\text{new}(\Gamma_0(2))) = k-1-\lfloor k/4 \rfloor - 2 \lfloor k/3 \rfloor - \delta(k/2)$ where $\delta(n) = 1$ if $n = 1$ else equals zero.

Suppose $p = 2$ and $k$ is odd. In this case, $S_k(\Gamma_1(4)) = S_k^\text{new}(\Gamma_0(4), \chi)$ where $\chi$ is the nontrivial Dirichlet character of $(\bb Z / 4 \bb Z)^*$. By \cite[Theorem 3]{MR369263}  each reciprocal root of $det(1 - U_2 s \mid S_k^\text{new}(\Gamma_0(4, \chi)))$ has absolute value equal to $2^{(k-1)/2}$.

The case $p = 3$ is similar. In this case, for $k$ even we may decompose $S_k(\Gamma_1(3)) = W \oplus S_k^\text{new}(\Gamma_1(3))$ where
\[
W := S_k(SL_2(\bb Z)) \oplus V_3 S_k(SL_2(\bb Z)).
\]
Using the same argument as above, every reciprocal root of $det(1 - U_3 s \mid W) = det(1 - T_3 s + 3^{k-1} s^2 \mid SL_2(\bb Z))$ has absolute value equal to $3^{(k-1)/2}$, and by \cite[Theorem 3]{MR369263} each reciprocal root of $det(1 - U_3 s \mid S_k^\text{new}(\Gamma_1(3)))$ equals $\pm 3^{(k-2)/2}$.

When $p = 3$ and $k$ is odd, then $S_k(\Gamma_1(3)) = S_k^\text{new}(\Gamma_0(3), \chi)$ where $\chi$ is the nontrivial Dirichlet character of $(\bb Z / 3 \bb Z)^*$. By \cite[Theorem 3]{MR369263}  each reciprocal root of $det(1 - U_3 s \mid S_k^\text{new}(\Gamma_0(3), \chi))$ has absolute value equal to $3^{(k-1)/2}$.
\end{proof}

\section{$p$-adic symmetric power $L$-functions}\label{S: padicSym}

In \cite{MR2385246}, a cohomology theory was introduced for the $p$-adic $\kappa$-symmetric power $L$-function $L(Sym^{\infty, \kappa} Kl, s)$ for $p \geq 5$. In this section we show the same theory holds for $p \geq 2$ without any changes. To align with \cite{MR2385246} we will work over the base field $\bb F_q$, where $q$ is a power of $p$. Define the Kloosterman exponential sum
\[
Kl_m(q, t) := \sum_{x \in \bb F_{q^{m \text{deg}(t)}}^*} \zeta_p^{Tr(x + \frac{t}{x})} \in \bb R
\]
where $\deg(t) := [\bb F_q(t) : \bb F_q]$. As before, the associated $L$-function is a quadratic polynomial 
\[
L(\bb F_q, t, s) := \exp\left( \sum_{m=1}^\infty Kl_m(q, t) \frac{s^m}{m} \right) = (1 - \pi_0(t) s)(1 - \pi_1(t) s)
\]
whose reciprocal roots satisfy $| \pi_0(t) | = | \pi_1(t) | = \sqrt{q^{\deg(t)}}$ and $\pi_0(t) \pi_1(t) = q^{\deg(t)}$. For $\kappa \in \bb Z_p$ define the $p$-adic $\kappa$-symmetric power $L$-function 
\[
L(Sym^{\infty, \kappa} Kl / \bb F_q, s) := \prod_{t \in | \bb G_m / \bb F_q |} \prod_{i = 0}^\infty \frac{1}{1 - \pi_0(t)^{\kappa-i} \pi_1(t)^i s^{\deg(t)}}.
\]
where the first product runs over all closed points of the algebraic torus defined over $\bb F_q$. 

\begin{theorem}\label{T: sym inf}
Let $p \geq 2$ be a prime number, and $\kappa \in \bb Z_p$. Then there exists $p$-adic spaces $H_\kappa^0$ and $H_\kappa^1$, and a completely continuous map $\bar \beta_{\infty, \kappa}$, such that
\[
L(Sym^{\infty, \kappa} Kl / \bb F_q, s) = \frac{det(1 - s \bar \beta_{\infty, \kappa} \mid H_\kappa^1)}{det(1 - q s \bar \beta_{\infty, \kappa}  \mid H^0_\kappa)}.
\]
If $\kappa \not= 0$ then $H_\kappa^0 = 0$, else when $\kappa = 0$ then dim $H_\kappa^0 = 1$ and $det(1 - q s \bar \beta_{\infty, \kappa}  \mid H^0_\kappa) = (1 - qs)$. Writing $L(Sym^{\infty, \kappa} Kl / \bb F_q, s) = \sum_{n \geq 0} c_n s^n$, then the $q$-adic Newton polygon satisfies
\[
ord_q c_n \geq n(n-1) \quad \text{for all } n \geq 0.
\]
When $p = 2$ and $\kappa \in 1 + 2 \bb Z_2$, then the 2-adic slopes of the reciprocal roots of $L(Sym^{\infty, \kappa} Kl, s)$ equal $2n$ for $n \geq 0$.
\end{theorem}

\begin{proof}
We refer the reader to \cite{MR2385246} for definitions and notation, specifically \cite[Section 4]{MR2385246}. In loc.cit., the only restriction on the prime $p$ came from the definition of $[\alpha_{\infty, a}]_\kappa$, and so the the cohomology theory in \cite[Section 4.1]{MR2385246} holds for all $p \geq 2$, which includes the result on $H_\kappa^0$. 

We now move to the map $[\bar \alpha_{\infty, a}]_\kappa$ in \cite[Section 4.3]{MR2385246}. First, we note that \cite[Lemma 4.2]{MR2385246} holds for all $p \geq 2$, with \cite[Theorem 3.5]{MR3572279} being one such proof. It follows that for all $p \geq 2$, $[\bar \alpha_{\infty, a}]_\kappa: \c S(\hat b, 1) \rightarrow \c S(\hat b/q, 1)$, where $\hat b := \frac{p}{p-1}$ and $\varepsilon = \hat b - \frac{1}{p-1} = 1$. Consequently, $\beta_{\infty, \kappa}$ is a well-defined endomorphism of $\c S(\hat b, 1)$ for all $p \geq 2$, and so 
\[
L(Sym^{\infty, \kappa} Kl / \bb F_q, s) = \frac{det(1 - s \bar \beta_{\infty, \kappa} \mid H_\kappa^1)}{det(1 - q s \bar \beta_{\infty, \kappa}  \mid H^0_\kappa)}.
\]
We move now to the $q$-adic Newton polygon. Suppose $\kappa \not= 0$. Here we follow the proof of \cite[Theorem 4.9]{MR2385246}. We have shown that $\beta_{\infty, \kappa, 1}$ is well-defined for all $p \geq 2$. Since $\gamma^{2n} t^n \in \c R(\hat b/p; 0) \subset \c S(\hat b/p, 1; 0)$, we have $[\bar \alpha_{\infty, 1}]_{\kappa}(\gamma^{2n} t^n) \in \c S(\hat b/p, 1; 0)$, and so $\beta_{\infty, \kappa, 1}(\gamma^{2n} t^n) \in \c R(\hat b, 1; 0)$. Thus by \cite[Theorem 4.7]{MR2385246}, $\bar \beta_{\infty, \kappa, 1}(\gamma^{2n}t^n) \in \c R(\hat b, 1; 0)$. This means, writing $\bar \beta_{\infty, \kappa, 1}(\gamma^{2n}t^n) = \sum_{m \geq 0} B(m,n) t^m = \sum_{m \geq 0} ( B(m,n) \gamma^{-2m}) \gamma^{2m} t^m$, we have
\[
ord_p ( B(m,n) \gamma^{-2m}) \geq 2\hat b m - \frac{2m}{p-1} = 2m.
\]
The lower bound on the $q$-adic Newton polygon now follows from the proof of \cite[Theorem 4.9]{MR2385246}.

Suppose now that $p = 2$ and $\kappa \in 1 + 2 \bb Z_2$. By Theorem \ref{T: overconv interp} and \cite{MR2106238} (see \cite{MR2135280} for a generalization), the lower bound is achieved:  $ord_p a_n = n(n-1)$ for every $n$.
\end{proof}

\bibliographystyle{amsplain}
\bibliography{../References/References.bib}

\end{document}